%% file: lamps.tex
\documentclass[a4paper]{article}
\usepackage{amsthm,amssymb,amsmath,enumerate,graphicx,epsf,bbold}

\newcommand{\COLORON}{0}
\newcommand{\NOTESON}{0}
\newcommand{\Debug}{0}
\input{defs}

\newcommand{\llr}{lamplighter}
\newcommand{\llg}{\llr\ graph}

\newcommand{\lapr}{\ensuremath{\wr}}
\newcommand{\rede}{switching edge}

\title{Lamplighter graphs do not admit harmonic functions of finite energy}

\author{Agelos Georgakopoulos\thanks{Supported by FWF grant P-19115-N18.} \medskip \\
  {Technische Universit\"at Graz}\\
  {Steyrergasse 30, 8010}\\
  {Graz, Austria}\\
}

\date{}
\begin{document}
\maketitle

\begin{abstract}
We prove that a \llg\ of a \lf\ graph over a finite graph does not admit a non-constant harmonic function of finite Dirichlet energy.
\end{abstract}

\begin{section}{Introduction}

The wreath product $G \wr H$ of two groups $G,H$ is a well-known concept. Cayley graphs of $G \wr H$ can be obtained in an intuitive way by starting with a Cayley graph of $G$ and associating with each of its vertices a lamp whose possible states are indexed by the elements of $H$, see below. Graphs obtained this way are called \llg s. A well-known special case are the Diestel-Leader \cite{DiLeCon} graphs $DL(n,n)$.

Kaimanovich and Vershik \cite[Sections 6.1, 6.2]{KaVeRan} proved that \llg s of infinite grids $\Z^d, d\geq 3$ admit non-constant, bounded, harmonic functions. Their construction had an intuitive probabilistic interpretation related to random walks on these graphs, which triggered a lot of further research on \llg s. For example, spectral properties of such groups are studied in \cite{DiSchiSpe,GrZuLam, LNWspe} and other properties related to random walks are studied in \cite{ErschDri,ErschGen,PiSaRan}. Harmonic functions on \llg s and the related Poisson boundary are further studied e.g.\ in \cite{BrWoPos,KaWoPoi,Sava}. Finally, Lyons, Pemantle and Peres  \cite{LPPRan} proved that the lamplighter graph of \Z\ over $\Z_2$ has the surprising property that  random walk with a drift towards a fixed vertex can move outwards faster than simple random walk.

It is known that the existence of a non-constant harmonic function of finite Dirichlet energy implies the existence of a non-constant bounded harmonic function \cite[Theorem~3.73]{soardi}. Given the aforementioned impact that bounded harmonic functions on \llg s have had, it suggests itself to ask whether these graphs have non-constant harmonic functions of finite Dirichlet energy.  For \llg s on a grid it is known that no such harmonic functions can exist, since  the corresponding groups are amenable and thus admit no non-constant harmonic functions of finite Dirichlet energy \cite{MeSoExt}. A.~Karlsson (oral communication) asked whether this is also the case for graphs of the form $T \wr \Z_2$ where $T$ is any regular tree. In this paper we give an affirmative answer to this question. In fact, the actual result is much more general:

\begin{theorem}\label{main}
Let $G$ be a connected \lfg\ and let $H$ be a connected finite graph with at least one edge. Then $G \lapr H$ does not admit any non-constant harmonic function of finite Dirichlet energy. 
\end{theorem} 

Indeed, we do not need to assume that any of the involved graphs is a Cayley graph. Lamplighter graphs on general graphs can be defined as in the usual case when all graphs are Cayley graphs; see the next section.

As an intermediate step, we prove a result (\Lr{lilem} below) that strengthens a theorem of Markvorsen,  McGuinness and Thomassen \cite{CThyperb} and might be applicable in order to prove that other classes of graphs do not admit non-constant Dirichlet-finite harmonic functions.
\end{section}

\begin{section}{Definitions}
We will be using the terminology of Diestel \cite{diestelBook05}. For a finite path $P$ we let $|P|$ denote the number of edges in $P$. For a graph \g and a set $U\subseteq V(G)$ we let $G[U]$ denote the subgraph of \g induced by the vertices in $U$. If \g is finite then its \defi{diameter} $diam(G)$ is the maximum distance, in the usual graph metric, of two vertices of $G$. 

Let \G, $H$ be connected graphs, and suppose that every vertex of \g has a distinct lamp associated with it, the set of possible states of each lamp being the set of vertices $V(H)$ of $H$. At the beginning all lamps have the same state $s_0\in V(H)$, and a ``\defi{\llr}'' is standing at some vertex of \G. In each unit of time the \llr\ is allowed to choose one of two possible moves: either walk to a vertex of $G$ adjacent to the vertex $x\in V(G)$ he is currently at, or switch the current state $s\in V(H)$ of $x$ into one of the states $s'\in V(H)$ adjacent with $s$. The \defi{\llg\ $G \lapr H$} is, then, a graph whose vertices correspond to the possible configurations of this game and whose edges correspond to the possible moves of the lamplighter. More formally, the vertex set of $G \lapr H$ is the set of pairs $(C,x)$ where $C: V(G) \to V(H)$ is an assignment of states \st\ $C(v)\neq s_0$ holds for only finitely many vertices $v\in V(G)$, and $x$ is a vertex of $G$ (the current position of the \llr). Two vertices $(C,x)$ and $(C',x')$ of $G \lapr H$ are joined by an edge if (precisely) one of the following conditions holds:
\begin{itemize}
\item $C=C'$ and $xx'\in E(G)$, or 
\item $x=x'$, all vertices except $x$ are mapped to the same state by $C$ and $C'$, and $C(x)C'(x)\in E(H)$.
\end{itemize}

This definition of $G \lapr H$ coincides with that of Erschler \cite{ErschGen}.

The \defi{blow-up} of a vertex $v\in V(G)$ in $L=G \lapr H$ is the set of vertices of $L$ of the form $(C,v)$. Similarly, the blow-up of a subgraph $T$ of $G$ is the subgraph of $L$ spanned by the blow-ups of the vertices of $T$. Given a vertex $x\in V(L)$ we let \defi{$[x]$} denote the vertex of $G$ the blow-up of which contains $x$. 
 
An edge of $L$ is a \defi{\rede} if it corresponds to a move of the lamplighter that switches a lamp; more formally, if it is of the form $(C,v)(C',v)$. For a \rede\ $e\in E(L)$ we let $[e]$ denote the corresponding edge of $H$.
A \defi{ray} is a $1$-way infinite path;  a $2$-way infinite path is called 
a \defi{double ray}. A \defi{tail} of a ray $R$ is an infinite (co-final) subpath of $R$.

A function $\phi:V(G) \to \R$ is \defi{harmonic}, if \fe\ $x\in V(G)$ \tho\ $\phi(x) = \frac{1}{d(x)}\sum_{xy\in E(G)} \phi(y)$, where $d(x)$ is the number of edges incident with $x$. Given such a function $\phi$, and an edge $e=uv$, we let $w_\phi(e):= (\phi(u) - \phi(v))^2$ denote the \defi{energy} dissipated by $e$. The \defi{(Dirichlet) energy} of $\phi$ is defined by $W(\phi):= \sum_{e\in E(G)} w_\phi(e)$.

\end{section}

\begin{section}{Proof of \Tr{main}}

We start with a lemma that might be applicable in order to prove that other classes of graphs do not admit non-constant Dirichlet-finite harmonic functions. This strengthens a result of \cite[Theorem~7.1]{CThyperb}.

\begin{lemma} \label{lilem}
Let \g be a connected \lfg\ \st\ \fe\ two disjoint rays $S,Q$ in \g \ti\ a constant $c$ and a sequence \seq{P} of pairwise edge-disjoint \pths{S}{Q}\ \st\ $|P_i| \leq c i$. Then \g does not admit a non-constant harmonic function of finite energy.
\end{lemma}
\begin{proof}

Let \g be a \lfg\ that admits a non-constant harmonic function $\phi$  of finite energy; it suffices to find two rays $S,Q$ in \g that do not satisfy the condition in the assertion.

Since $\phi$ is non-constant, we can find an edge $x_0 x_1$ satisfying $\phi(x_1) > \phi(x_{0})$. By the definition of a harmonic function, it is easy to see that $x_0 x_1$ must lie in  a double ray $D = \ldots x_{-1} x_0 x_1 \ldots$ \st\ $\phi(x_i) \geq \phi(x_{i-1})$ \fe\ $i\in\Z$; indeed, every vertex $x\in V(G)$ must have a neighbour $y$ \st\ $\phi(y) \geq \phi(x)$. 

Define the sub-rays $S = x_0 x_1 x_2 \ldots$ and $Q = x_0 x_{-1} x_{-2} \ldots$ of $D$. Now suppose \ti\ a sequence \seq{P} of pairwise edge-disjoint \pths{S}{Q}\ \st\ $|P_i| \leq c i$ for some constant $c$. 

Note that by the choice of $D$ \ti\ a bound $u>0$ \st\ $u_i:= |\phi(s_i) - \phi(q_i)|\geq u$ \fe\ $i$, where $s_i\in V(S)$ and $q_i\in V(Q)$ are the endvertices of $P_i$.

For every edge $e= xy$ let $f(e):= |\phi(y) - \phi(x)|$. 
Let $X_i$ be the set of edges $e$ in $P_i$ such that $f(e)\ge
0.9 \frac{u}{ci}$, and let $Y_i$ be the set of all other edges in $P_i$. As $|P_i|\leq c i$ by assumption, the edges in $Y_i$ contribute less than $0.9 u$ to $u_i$, thus $\sum_{e\in X_j} f(e) > 0.1 u$ must hold. But since $f(e)\geq 0.9 \frac{u}{ci}$ for every $e\in X_j$, we have $\sum_{e\in X_j} w_\phi(e) >0.1 \times 0.9 \frac{u^2}{ci}$. As the sets $X_j$ are pairwise edge-disjoint, and as the series $\sum_i 1/i$ is not convergent, this contradicts the fact that $\sum_{e\in E(G)} w_\phi(e)$ is finite. 
\end{proof}

We now apply \Lr{lilem} to prove our main result.

\begin{proof}[Proof of \Tr{main}]
We will show that $L:= G \lapr H$ satisfies the condition of \Lr{lilem}, from which then the assertion follows. So let $S,Q$ be any two disjoint rays of $L$. 

Since $L$ is connected we can find a double ray $D$ in $L$ that contains a tail $S'$ of $S$ and a tail $Q'$ of $Q$. Let $s_0$ (respectively, $q_0$) be the first vertex of $S'$ (resp.\ $Q'$). Let $V_0$ be the set of vertices of $G$ the blow-up of which meets the path $s_0 D q_0$. Note that $V_0$ induces a connected subgraph of \G, because the lamplighter only moves along the edges of $G$. Thus we can choose a spanning tree $T_0$ of $G[V_0]$. 

For $i=1,2,\ldots$ we construct an \pth{S'}{Q'} $P_i$ as follows. Let $s_i$ be the first vertex of $S'$ not in the blow-up of $V_{i-1}$, and let $q_i$ be the first vertex of $Q'$ not in the blow-up of $V_{i-1}$. Let $V_i:= V_{i-1} \cup \{s_i,q_i\}$, and extend $T_{i-1}$ into a spanning tree $T_i$ of $G[V_i]$ by adding two edges incident with  $s_i$ and  $q_i$ respectively; such edges do exist: their blow-up contains the edges of $S', Q'$ leading into $s_i,q_i$ respectively. 

We now construct an \pth{s_i}{q_i}\ $P_i$. Pick a \rede\ $e= s_i s'_i$ incident with $s_i$. Then let $X_i$ be the unique path in $L$ from  $s'_i$ to a vertex $q^+_i$ with $[q^+_i] = [q_i]$ \st\ $X_i$ is contained in the blow-up of $T_i$. Pick a \rede\ $f = q^+_i q^-_i$ incident with $q^+_i$. Then follow the unique path $Y_i$ in $L$ from $q^-_i$ to a vertex $s^+_i$ with $[s^+_i]=[ s_i]$ \st\ $Y_i$ is contained in the blow-up of $T_i$. Let $e'= s^+_i s^-_i$ be the \rede\ incident with $s^+_i$ \st\ $[e']=[ e]$. Finally, let $Z_i$ be a path from $s^-_i$ to the unique vertex $q'_i$ with $[q_i q'_i ]=[f]$, \st\ the interior of $Z_i$ is contained in the blow-up of $V_{i-1}$ and $Z_i$ has minimum length under all paths with these properties. Such a path exists because every lamp at a vertex in $G - V_{i-1}$ has the same state in $s^-_i$ and $q'_i$; indeed, the lamps in $G - V_{i}$ were never switched in the above construction, the lamp at $[s_i]$ was switched twice on the way from $s_i$ to $s^-_i$ using the same \rede\ $[e]$, which means that its state in both endpoints of $Z_i$ coincides with that in $s_i$ and $q_i$, and finally the lamp at $[q'_i]$ has the same state in both endpoints of $Z_i$, namely the state $[f]$ leads to. Now set $P_i:= s_i s'_i X_i q^+_i q^-_i Y_i s^+_i s^-_i Z_i q'_i q_i$.

It is not hard to check that the paths $P_i$ are pairwise disjoint. Indeed, let $i<j\in \N$. Then, by the choice of the vertices $s_j, q_j$ and the construction of $P_j$, it follows that for every inner vertex $x$ of $P_j$, the configuration of $x$ differs from the configuration of any vertex in $P_i$ in at least one of the two lamps at $[s_j]$ and $[q_j]$.

It remains to show that \ti\ a constant $c$ \st\ $|P_i| \leq c i$ \fe\ $i$. To prove this, 
note that $|P_i| = |X_i| + |Y_i| + |Z_i| + 4$; we will show that the latter three subpaths grow at most linearly with $i$, which then implies that this is also true for $P_i$.

Firstly, note that $diam(T_{i}) - diam(T_{i-1}) \leq 2$ since $V(T_i):= V(T_{i-1}) \cup \{s_i,q_i\}$. By the choice of $X_i$ we have $|X_i| \leq diam(T_{i})$, from which follows that \ti\ a constant $c_1$ \st\ $|X_i| \leq c_1 i$. By the same argument, we have $|Y_i| \leq c_1 i$. 

It remains to bound the length of $Z_i$. For this, note that if $T$ is a finite tree and $v,w\in V(T)$, then there is a $v$-$w$~walk $W$ in $T$ containing all edges of $T$ and satisfying $|W|\leq 3 |E(T)|$; indeed, starting at $v$, one can first walk around the ``perimeter'' of $T$ traversing every edge precisely once in each direction ($2 |E(T)|$ edges), and then move ``straight'' from $v$ to $w$ (at most $|E(T)|$ edges). Thus, in order to choose $Z_i$, we could put a lamplighter at the vertex and configuration indicated by $s^-_i$, and let him move in $T_i \subset G$ along such a walk $W$ from $[s^-_i]$ to $[q'_i]$, and every time he visits a new vertex $x$ let him change the state of $x$ to the state indicated by $q'_i$. This bounds the length of $Z_i$ from above by $3|E(T_i)| diam(H)$, and since $|E(T_i)|-|E(T_{i-1})|=2$ and $H$ is fixed, we can find a constant $c_2$ \st\ $|Z_i| \leq c_2 i$ \fe\ $i$. This completes the proof that $|P_i|$ grows at most linearly with $i$.

Thus we can now apply \Lr{lilem} to prove that $G \lapr H$ does not admit a non-constant harmonic function of finite energy.
\end{proof}

\begin{problem}
 Does the assertion of \Tr{main} still hold if $H$ is an infinite \lfg?
\end{problem}

\Lr{lilem} might be applicable in order to prove that other classes of graphs do also not admit non-constant Dirichlet-finite harmonic functions. For example, it yields an easy proof of the (well-known) fact that infinite grids  have this property.
\end{section}

\bibliographystyle{plain}
\bibliography{collective}
\end{document}

%% file: defs.tex
\usepackage[usenames,dvips]{color} 
\usepackage{amsthm,amssymb,amsmath,enumerate,graphicx,epsf,mathbbol,stmaryrd}
\usepackage[dvips, bookmarks, colorlinks=false]{hyperref}

\newcommand{\comment}[1]{}
\newcommand{\COMMENT}[1]{}

\definecolor{darkgray}{rgb}{0.3,0.3,0.3}
\newcommand{\defi}[1]{{\color{darkgray}\emph{#1}}}

\comment{
	\begin{lemma}\label{}	
\end{lemma}
\begin{proof}

\end{proof}

\begin{theorem}\label{}
\end{theorem} 
\begin{proof} 	

\end{proof}

}

\newtheorem{proposition}{Proposition}[section]

\newtheorem{theorem}[proposition]{Theorem}

\newtheorem{lemma}[proposition]{Lemma}

\newtheorem{conjecture}{{\color{red}Conjecture}}[section]

\newtheorem{problem}[conjecture]{{\color{red}Problem}}

\newtheorem{examp}[proposition]{Example}

\newcommand{\FIG}{0}

\ifnum \NOTESON = 1 \newcommand{\note}[1]{ 

	\ 

	{\color{blue} \hspace*{-60pt} NOTE: \color{Turquoise}{\small  \tt \begin{minipage}[c]{1.1\textwidth}  #1 \end{minipage} \ignorespacesafterend }} 
	
	\ 
	
	}
\else \newcommand{\note}[1]{} \fi

\ifnum \Debug = 1 
\else  \fi

\ifnum \FIG = 1 
\else  \fi

\ifnum \Debug = 1 \usepackage[notref,notcite]{showkeys}
\fi

\ifnum \COLORON = 0 \renewcommand{\color}[1]{}
\fi

\newcommand{\N}{\ensuremath{\mathbb N}}
\newcommand{\R}{\ensuremath{\mathbb R}}
\newcommand{\Z}{\ensuremath{\mathbb Z}}

\newcommand{\pth}[2]{\ensuremath{#1}\text{--}\ensuremath{#2}~path}

\newcommand{\pths}[2]{\ensuremath{#1}\text{--}\ensuremath{#2}~paths}

\newcommand{\seq}[1]{\ensuremath{(#1_i)_{i\in\N}}} 
 
\newcommand{\g}{\ensuremath{G\ }}
\newcommand{\G}{\ensuremath{G}}

\newcommand{\Lr}[1]{Lemma~\ref{#1}}
\newcommand{\Tr}[1]{Theorem~\ref{#1}}

\newcommand{\lf}{locally finite}
\newcommand{\lfg}{locally finite graph}

\newcommand{\fe}{for every}

\newcommand{\st}{such that}

\newcommand{\ti}{there is}
\newcommand{\tho}{there holds}

\newcommand{\mySection}[2]{}

